\documentclass[preprint,12pt]{elsarticle}

\usepackage{subfig}
\usepackage{graphicx}
\usepackage{caption,color}   
\usepackage{amssymb}
\usepackage{amsthm}
\usepackage{amsmath}
\usepackage{epic}
\usepackage{setspace}
\usepackage{float}
\usepackage{multirow}
\newtheorem{thm}{Theorem}[section]

\newtheorem{con}[thm]{Conjecture}
\newtheorem{lem}[thm]{Lemma}
\newtheorem{pro}[thm]{Proposition}

\numberwithin{equation}{section}
\journal{}

\begin{document}
\begin{spacing}{1.15}
\begin{frontmatter}
\title{\textbf{~High-ordered spectral characterizations of graphs\\ }}
\author{Lixiang Chen}
\author{Lizhu Sun}
\author{Changjiang Bu}\ead{buchangjiang@hrbeu.edu.cn}

\address{College of Mathematical Sciences, Harbin Engineering University, Harbin 150001, PR China}

\begin{abstract}
The spectrum of the $k$-power hypergraph of a graph $G$ is called the $k$-ordered spectrum of $G$.
If graphs $G_1$ and $G_2$ have same $k$-ordered spectrum for all positive integer $k\geq2$, $G_1$ and $G_2$ are said to be high-ordered cospectral.
If all graphs who are high-ordered cospectral with the graph $G$ are isomorphic to $G$, we say that $G$ is determined by the high-ordered spectrum.
In this paper, we use the high-ordered spectrum of graphs to study graph isomorphism and show that all Smith's graphs are determined by the high-ordered spectrum.
And we give infinitely many pairs of trees with same spectrum but different high-ordered spectrum by high-ordered cospectral invariants of trees,
it means that we can determine that these cospectral trees are not isomorphism by the high-ordered spectrum.
\end{abstract}

\end{frontmatter}

\section{Introduction}
If there was a ``Holy Grail'' in graph theory, it would be a practical test for graph isomorphism \cite{merris1994laplacian}.
In 1956, G\"{u}nthard and Primas used the spectrum of graphs to determine whether the graphs are isomorphic \cite{gunthard1956zusammenhang}.
In 1957, Collatz and Sinogowitz presented a pair of non-isomorphic cospectral tree \cite{von1957spektren}.
Subsequently,
many non-isomorphic \textit{cospectral} graphs (i.e., graphs having the same spectrum) were found \cite{cvetkovic1971graphs,godsil1982constructing,schwenk1973almost}.
For instance, two graphs shown in Figure \ref{tu1} are non-isomorphic cospectral graphs given by Cvetkovi\'{c} in 1971 \cite{cvetkovic1971graphs}.
This pair of graphs  is usually called the \textit{Saltire pair} because the two graphs superposed give the Scottish flag: Saltire  \cite{van2003graphs}.

\begin{figure}[H]
  \centering
 \subfloat[$G_1$]{
\begin{minipage}[t]{0.3\textwidth}
\centering
\includegraphics[scale=0.3]{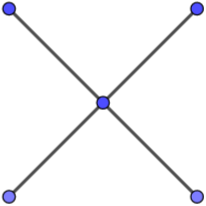}
\end{minipage}
} \subfloat[$G_2$]{
\begin{minipage}[t]{0.3\textwidth}
\centering
\includegraphics[scale=0.3]{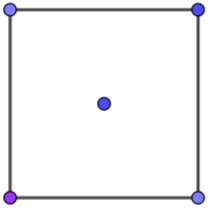}
\end{minipage}
}
 \caption{The Saltire pair}\label{tu1}
\end{figure}

If all graphs who are cospectral with a graph $G$ are isomorphic to $G$, the graph $G$ is said to be \textit{determined by the spectrum} (DS for short).
Until now, the known DS graphs are very special and most of the techniques involved in proving them to be DS cannot be applied to general graphs \cite{cvetkovic1980spectra,van2003graphs,van2009developments,Wang2017simple}.

In this paper, we use the high-ordered spectrum of graphs to study the high-ordered spectral characterizations of graphs.
Since two graphs are isomorphic if and only if their power hypergraphs are isomorphic and the spectrum of the power hypergraph contains more graph structure information, the high-ordered spectrum of graphs can determine more graphs than the spectrum of graphs.
For example, the Saltire pair can not be determined by the spectrum but can be determined by the high-ordered spectrum (shown in Section 3).

%
Let the $k$-power hypergraph $G^{(k)}$ be the $k$-uniform hypergraph that is obtained by adding $k-2$ new vertices to each edge of $G$.
The spectrum of  $G^{(k)}$ is called the \textit{$k$-ordered spectrum} of $G$.
If graphs $G_1$ and $G_2$ have same $k$-ordered spectrum for all positive integer $k\geq2$, $G_1$ and $G_2$ are said to be \emph{high-ordered cospectral}.
If all graphs who are high-ordered cospectral with the graph $G$ are isomorphic to $G$, we say that $G$ is \emph{determined by the high-ordered spectrum} (\textit{DHS} for short).
If graphs which have same $k$-ordered spectrum with the graph $G$ are isomorphic to $G$, we say that $G$ is \textit{determined by the $k$-ordered spectrum}.
Obviously, the graphs determined by the $k$-ordered spectrum are DHS.
The graphs determined by the spectrum ($2$-ordered spectrum) are DHS.

In 1970, Smith classified all connected graphs with spectral radius at most $2$ \cite{smith1970some}, which usually are called ``Smith's graphs".
In 2009, Van Dam and Haemers showed that not all Smith's graphs are determined by the spectrum  \cite{van2009developments}.
We show that all Smith's graphs are determined by the high-ordered spectrum in this paper.
In 1973, Schwenk gave a useful method to construct cospectral trees and proved his famous conclusion:``Almost all trees are not DS" \cite{schwenk1973almost}.
We show that the high-ordered spectrum of infinitely many pairs of cospectral trees constructed by Schwenk's method are different.




This paper is organized as follows.
In Section 2, we introduce  the spectrum of hypergraph and some lemmas.
In Section 3, we show that Smith's graphs are determined by the high-ordered spectrum.
In Section 4, we give some high-ordered cospectral invariants about the number of some subtrees.
As their applications, we give infinitely many pairs of cospectral trees with different high-ordered spectrum.
In Section 5, conclusions and future work are given.

\section{Preliminaries}

In this section, we introduce  the spectrum of hypergraph and some lemmas which are important to the subsequent results in the paper.
For a positive integer $n$, let $\left[ n \right] = \left\{ {1, \ldots ,n} \right\}$.
A $k$-order $n$-dimension complex tensor $T= \left( {{t_{{i_1} \cdots {i_k}}}} \right) $ is a multidimensional array with $n^k$ entries on complex number field $\mathbb{C}$, where ${i_j} \in \left[ n \right]$, $j = 1, \ldots ,k$.
Denote the set of $n$-dimension complex vectors by $\mathbb{C}^n$.
For $x = {\left( {{x_1}, \ldots ,{x_n}} \right)^{\rm{T}}} \in {\mathbb{C}^n}$, $T{x^{k - 1}}$ denotes a vector in $\mathbb{C}^{n}$ whose $i$-th component is
\[{\left( {T{x^{k - 1}}} \right)_i} = \sum\limits_{{i_2}, \ldots ,{i_k}=1}^n {{t_{i{i_2} \cdots {i_k}}}{x_{{i_2}}} \cdots {x_{{i_k}}}} .\]
If there exist $\lambda \in \mathbb{C}$ and a nonzero vector $x =(x_1,x_2,\ldots,x_n)^{\mathrm{T}}\in {\mathbb{C}^n}$ such that $T{x^{k - 1}} = \lambda {x^{\left[ {k - 1} \right]}}$, then $\lambda$ is called an \emph{eigenvalue} of $T$ and $x$ is an \emph{eigenvector} of $T$ corresponding to $\lambda$, where ${x^{\left[ {k - 1} \right]}} = {\left( {x_1^{k - 1}, \ldots ,x_n^{k - 1}} \right)^{\rm{T}}}$ \cite{lim2005singular,qi2005eigenvalues}.

A hypergraph $H=(V(H),E(H))$ is called \emph{$k$-uniform} if each edge of $H$ contains exactly $k$ vertices.
For a $k$-uniform hypergraph $H$ with $n$ vertices, its \emph{adjacency tensor} ${A}_H=(a_{i_1i_2\ldots i_k})$ is a $k$-order $n$-dimension tensor, where
\[{a_{{i_1}{i_2} \ldots {i_k}}} = \left\{ \begin{array}{l}
 \frac{1}{{\left( {k - 1} \right)!}},{\kern 37pt}\mathrm{ if}{\kern 2pt}{ \left\{ {{i_1},{i_2},\ldots ,{i_k}} \right\} \in {E(H)}}, \\
 0, {\kern 57pt}\mathrm{ otherwise}. \\
 \end{array} \right.\]
All the eigenvalues of ${A}_H$ are called the \textit{spectrum} of the hypergraph $H$ \cite{cooper2012spectra}.
When $k=2$, ${A}_H$ is the adjacency matrix of the graph $H$.

The $k$-ordered spectrum of graph $G$ is the spectrum of power hypergraph $G^{(k)}$.
In \cite{Zhou2014Some},  Zhou et al. show that eigenvalues of $G^{(k)}$ can be obtained from the eigenvalues of $G$.
In \cite{cardoso2020spectrum}, the authors gave all the eigenvalues of $G^{(k)}$ without counting multiplicity as following.
\begin{lem}\cite{cardoso2020spectrum}\label{yinli1}
Let $G^{(k)}$ be the $k$-power hypergraph of a graph $G$.
\\
(1) When $k=3$, $\lambda$ is an eigenvalue of $G^{(3)}$  if and only if there is an induced subgraph with an eigenvalue $\beta$ such that $\beta^2=\lambda^k$.\\
(2) When $k>3$, $\lambda$ is an eigenvalue of $G^{(k)}$ if and only if there is a subgraph with an eigenvalue $\beta$ such that $\beta^2=\lambda^k$.
\end{lem}

In \cite{chen2021reduction}, the authors gave the characteristic polynomial and then give all distinct eigenvalues of power hyperpaths.
\begin{lem}\cite{chen2021reduction}\label{yinlichen}
All distinct elements in the set
$$ \left\{{\left( {2\cos \frac{\pi }{{j+1 }}t} \right)^{\frac{2}{k}}}{e^{\mathbf{i}\frac{{2\pi}}{k}\theta }}:j\in[n],t\in[j],\theta\in[k]\right\}$$
are all distinct eigenvalues of the $k$-uniform hyperpath $P^{(k)}_n$, where $\mathbf{i}=\sqrt{-1}$.
\end{lem}

Since the connected subgraphs of the cycle $C_n$ are $C_n$ and $P_j$ for $j \in [n-1]$, we obtain all distinct eigenvalues of $C^{(k)}_n$ from all eigenvalues of $C_n$ and $P_j$ by Lemma \ref{yinli1}.
\begin{lem}\label{dingliben}
All distinct elements in the set
$$\left\{{\left( {2\cos \frac{\pi }{{j+1 }}t} \right)^{\frac{2}{k}}}{e^{\mathbf{i}\frac{{2\pi}}{k}\theta }}, {\left( {2\cos \frac{2\pi }{{n }}r} \right)^{\frac{2}{k}}}{e^{\mathbf{i}\frac{{2\pi}}{k}\theta }} :j\in[n-1],t\in[j],r \in [n],\theta\in[k]\right\} $$
are all distinct eigenvalues of the $k$-power hypercycle $C^{(k)}_n$ (for $k>3$), where $\mathbf{i}=\sqrt{-1}$.
\end{lem}
\begin{proof}
The connected subgraphs of cycle $C_n$ are $C_n$ and $P_j$ for $j \in [n-1]$.
By Lemma \ref{yinli1} and \ref{yinlichen}, we know that  all distinct numbers in the set $$\left\{{\left( {2\cos \frac{\pi }{{j+1 }}t} \right)^{\frac{2}{k}}}{e^{\mathbf{i}\frac{{2\pi}}{k}\theta }}:j\in[n-1],t\in[j],\theta\in[k]\right\} \bigcup \left\{{\lambda ^{\frac{2}{k}}}{e^{\mathbf{i}\frac{{2\pi}}{k}\theta }}: \lambda \in \sigma(C_n), \theta\in[k]\right\} $$
are all distinct eigenvalues of $C^{(k)}_n$,
where $\sigma(C_n)$ is the set of all eigenvalues of $C_n$.
Since $\sigma(C_n)=\{2\cos{\frac{2\pi r}{n}} :r=1,2,\ldots,n\}$ ( Page 72 in \cite{cvetkovic1980spectra}),
we know that all distinct numbers in the set
$$\left\{{\left( {2\cos \frac{\pi }{{j+1 }}t} \right)^{\frac{2}{k}}}{e^{\mathbf{i}\frac{{2\pi}}{k}\theta }}, {\left( {2\cos \frac{2\pi }{{n }}r} \right)^{\frac{2}{k}}}{e^{\mathbf{i}\frac{{2\pi}}{k}\theta }} :j\in[n-1],t\in[j],r \in [n],\theta\in[k]\right\} $$
are all distinct eigenvalues of $C^{(k)}_n$.
\end{proof}

The \textit{$d$-th order spectral moment} $\mathrm{S}_d(G)$ of a graph $G$ is the sum of $d$-th power of all eigenvalues of $G$.
Two graphs are cospectral if and only if their $d$-th order spectral moments are equal  \cite{van2009developments}.
The spectral moment of graph is an important parameter in the topic of graph spectral characterizations \cite{2008The,cvetkovic1987spectra,lepovic2002no,van2003graphs}.
Let $\widehat{G}$ be a connected subgraph of $G$.
We use $N_{G}(\widehat{G})$ to denote the number of subgraphs of $G$ isomorphic to $\widehat{G}$,
$c_d(\widehat{G})$ to denote the number of closed walks with length $d$ in the graph $\widehat{G}$ running through all the edges at least once.
The $d$-th order spectral moment $\mathrm{S}_d(G)$ can be represented as a linear combination of the number of connected subgraphs \cite{2008The,cvetkovic1987spectra,cvetkovic1980spectra}, i.e.,
${\mathrm{S}_d}(G)= \sum\nolimits_{\widehat{G} \in G(d)} {{c_d}(\widehat{G}){N_G}( \widehat{G} )}$,
where $G(d)$ is the set of connected subgraphs of $G$ with at most $d$ edges.
The coefficient $c_d(\widehat{G})$ is called the \textit{$d$-th order spectral moment coefficient of $\widehat{G}$}.
In \cite{Chen2020Spectral}, the authors gave a formula for the $d$-th order spectral moment coefficients of trees.

\begin{lem}\cite{Chen2020Spectral}\label{yinlixs}
The  $d$-th order  spectral moment of the tree $T$ is
\[{\mathrm{S}_d}\left( T \right) = \left\{ \begin{array}{l}
\sum\limits_{m=1}^{{\frac{d}{2}}} {{\sum\limits_{\widehat{T} \in {\mathbf{T}({m})}} {c_{d}(\widehat{T})} } N_{T}(\widehat{T})},{\kern16pt}2\mid d, \\
0,{\kern120pt} 2 \nmid d, \\
 \end{array} \right.\]
where ${\mathbf{T}({m})}$ is the set of subtrees of $T$ with $m$ edges.
The $d$-th order spectral moment coefficient of the  subtree $\widehat{T}$ is
\begin{align}\label{pjxsgs}
c_{d}({\widehat{T}}) = \left\{ \begin{array}{l}
{d}{\sum\limits_{{\sum_{e \in E({\widehat{T}})}}w(e) = \frac{d}{2}}}\left({\prod\limits_{e \in E({\widehat{T}})}}w{(e)}{\prod _{v \in V({\widehat{T}})}}\frac{{({d_v} - 1)!}}{{{r_v}}}\right),{\kern16pt}2\mid d, \\
0,{\kern228pt} 2 \nmid d, \\
 \end{array} \right.
\end{align}
where $w(e)$ is a positive integer weight of $e \in E(\widehat{T})$, ${d_v} = {\sum _{ e \in E_v(\widehat{{T}})}}{w(e)}$, ${r_v} = \prod\nolimits_{{ e \in E_v(\widehat{{T}})} } {{w(e)}!}$.
\end{lem}


The $d$-th order spectral moment $\mathrm{S}_d(H)$ of a $k$-uniform hypergraph $H$ is the sum of $d$-th power of all eigenvalues of $H$, i.e., $\mathrm{S}_d(H)=\sum_{\lambda \in \sigma(H)}{\lambda^{d}}$, where $\sigma(H)$ is the spectrum of $H$.
Two hypergraphs are cospectral if and only if their $d$-th order spectral moments are equal for all $d\geq1$ \cite{clark2021harary}.
Clark and Cooper expressed the characteristic polynomial coefficients of a uniform hypergraph $H$ by the spectral moments of $H$ and gave the
``Harary-Sachs Theorem" of hypergraphs \cite{clark2021harary}.
In \cite{Chen2020Spectral}, the authors expressed the spectral moment of power hypertree by the number of subtrees as following.
\begin{lem}\cite{Chen2020Spectral}\label{yinlimcs}
Let $T^{(k)}$ be the $k$-power hypertree of a tree $T$. Let $c_{i}({\widehat{T}})$ denote the $i$-th order spectral moment coefficient of the subtree ${\widehat{T}}$.
Then the $d$-th spectral moment of $T^{(k)}$ is
\begin{align*}
{\mathrm{S}_d}\left( T^{(k)}\right) = \left\{ \begin{array}{l}
\sum\limits_{m=1}^{{\frac{d}{k}}}\frac{1}{2}{{(k - 1)}^{\left( {|E(T)| - m} \right)(k - 1)}}{k^{m(k - 2) + 1}} {\sum\limits_{\widehat{T} \in {\mathbf{T}(m)}}{c_{\frac{2d}{k}}(\widehat{T})} } N_{T}(\widehat{T}),{\kern6pt}k\mid d, \\
 0,{\kern280pt}k \nmid d ,\\
 \end{array} \right.
\end{align*}
 where $\mathbf{T}(m)$ is the set of subtrees of $T$ with $m$ edges.
\end{lem}


\section{All Smith's graphs are DHS}
Connected graphs with spectral radius at most $2$, as shown as in Figure \ref{x1}, usually are called ``Smith's graphs", since Smith classified all connected graphs with spectral radius at most $2$ in 1970 \cite{smith1970some}.
In 2009, Van Dam and Haemers showed that not all Smith's graphs are determined by the spectrum  \cite{van2009developments}.
In this section, we show that all Smith's graphs are DHS.

Smith's graphs are determined by the spectrum except for the graphs $\widetilde{D}_n$ and $\widetilde{E}_6$ \cite{van2009developments},
then all Smith's graphs are DHS if and only if  $\widetilde{D}_n$ and $\widetilde{E}_6$ are DHS.
In order to prove that  $\widetilde{D}_n$ and $\widetilde{E}_6$ are DHS, we only need to give all non-isomorphic cospectral graphs of $\widetilde{D}_n$ (resp. $\widetilde{E}_6$) and then prove that these cospectral graphs are not high-ordered cospectral with $\widetilde{D}_n$ (resp. $\widetilde{E}_6$).


\begin{figure}[H]
\centering
\subfloat[$P_n$]{
\begin{minipage}[t]{0.20\textwidth}
\centering
\includegraphics[scale=0.3]{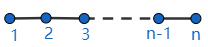}
\end{minipage}
}
\subfloat[$D_n$]{
\begin{minipage}[t]{0.20\textwidth}
\centering
\includegraphics[scale=0.3]{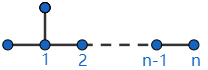}
\end{minipage}
}
\subfloat[$E_6$]{
\begin{minipage}[t]{0.20\textwidth}
\centering
\includegraphics[scale=0.3]{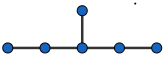}
\end{minipage}
}
\subfloat[$E_7$]{
\begin{minipage}[t]{0.20\textwidth}
\centering
\includegraphics[scale=0.3]{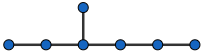}
\end{minipage}
}\\
\subfloat[$E_8$]{
\begin{minipage}[t]{0.20\textwidth}
\centering
\includegraphics[scale=0.3]{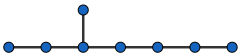}
\end{minipage}
}
\subfloat[$C_n$]{

\begin{minipage}[t]{0.20\textwidth}
\centering
\includegraphics[scale=0.3]{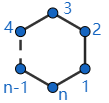}
\end{minipage}
}
\subfloat[$\widetilde{D}_n$]{
\begin{minipage}[t]{0.20\textwidth}
\centering
\includegraphics[scale=0.3]{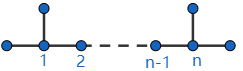}
\end{minipage}
}
\subfloat[$\widetilde{E}_6$]{
\begin{minipage}[t]{0.25\textwidth}
\centering
\includegraphics[scale=0.3]{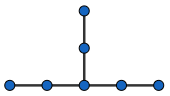}
\end{minipage}
}\\
\subfloat[$\widetilde{E}_7$]{
\begin{minipage}[t]{0.25\textwidth}
\centering
\includegraphics[scale=0.3]{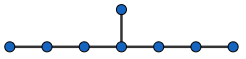}
\end{minipage}
}
\subfloat[$\widetilde{E}_8$]{
\begin{minipage}[t]{0.40\textwidth}
\centering
\includegraphics[scale=0.3]{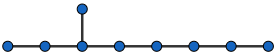}
\end{minipage}
}
\caption{Smith's graphs}\label{x1}
\end{figure}

We use  cospectral invariants of graphs to give all non-isomorphic cospectral graphs of graph $\widetilde{D}_n$ (resp. $\widetilde{E}_6$).
In \cite{cvetkovic1987spectra}, Cvetkovi\'{c} and Rowlinson gave the following cospectral invariant of graphs about the number of subgraphs.

\begin{lem}\cite{cvetkovic1987spectra}\label{yinli2}
If graphs $G$ and $G^*$ are cospectral, then $N_{P_3}(G)+2N_{C_4}(G)=N_{P_3}(G^*)+2N_{C_4}(G^*)$.
\end{lem}

We give all non-isomorphic cospectral graphs of $\widetilde{D}_n$ (resp. $\widetilde{E}_6$) by the above cospectral invariant.

\begin{lem}\label{zjyl}
Graph $G$ is a non-isomorphic cospectral graphs of $\widetilde{D}_n$ (resp. $\widetilde{E}_6$) if and only if $G= C_4+P_{n}$ (resp. $C_6+K_1$).
\end{lem}

\begin{proof}
It is well know that $C_4+P_{n}$ is a non-isomorphic cospectral graphs of $\widetilde{D}_n$ (  Page 77 in \cite{cvetkovic1980spectra}).
Let $G$ be a non-isomorphic cospectral graph of $\widetilde{D}_n$.
We will show that $G=C_4+P_{n}$.


Since $G$ is cospectral with $\widetilde{D}_n$ and $G$ is not a Smith's graph, we know that the spectral radius of $G$ is equal to $2$ and  $G$ is not connected.
Then the spectral radii of every connected components of $G$ are at most $2$.
Since connected graphs with spectral radius at most $2$ are cycles or trees, every connected components of $G$ are cycles or trees.
From $|V(G)|=|V(\widetilde{D}_n)|=n+4$ and $|E(G)|=|E(\widetilde{D}_n)|=n+3$, we get $|E(G)|=|V(G)|-1$.
Then only one connected component of $G$ is a tree and the other connected components are cycles.
The algebraic multiplicity of the eigenvalue $\lambda=2$ of $G$ (or $\widetilde{D}_n$) is 1 ( Page 77 in \cite{cvetkovic1980spectra}),
it yields that only one connected component of $G$ is cycle and the spectral radii of the other connected components are less than $2$.
Then the graph $G$ contains exactly two connected components, one of them is a cycle and the other is a tree with spectral radius less than $2$.

We know that  $G=C_s+T_{n+4-s}$, where $C_s$ is a cycle with $s$ vertices and $T_{n+4-s}$ is a tree whose spectral radius is less than 2.
Since $G$ is cospectral with $\widetilde{D}_n$, we have $N_{P_3}(G)+2N_{C_4}(G)=N_{P_3}(\widetilde{D}_n)+2N_{C_4}(\widetilde{D}_n)$ by Lemma \ref{yinli2}.
We know that $N_{P_3}(\widetilde{D}_n)+2N_{C_4}(\widetilde{D}_n)=n+4$ from Figure \ref{x1}.
Then $N_{P_3}(C_s)+N_{P_3}(T_{n+4-s})+2N_{C_4}(C_s)+2N_{C_4}(T_{n+4-s})=n+4$.
We have $2N_{C_4}(C_s)=n+4-N_{P_3}(C_s)-N_{P_3}(T_{n+4-s})=n+4-s-N_{P_3}(T_{n+4-s})$.
Since the spectral radius of $T_{n+4-s}$ is less than 2, we know that  $T_{n+4-s}$ is isomorphic to one of  $P_{n+4-s}$, $D_{n+2-n}$, $E_6$, $E_7$ and $E_8$.
If $T_{n+4-s} =P_{n+4-s}$, we have $N_{P_3}(T_{n+4-s})=n+2-s$ when $n-s \geq -1$ and $N_{P_3}(T_{n+4-s})=0$ when $-3\leq n-s<-1$.
When $n-s \geq -1$,  we have $N_{C_4}(C_s)=1$, i.e., $s=4$.
When $-3\leq n-s<-1$, we have $1\leq2N_{C_4}(C_s)<3$.
Then $N_{C_4}(C_s)=1$, i.e., $s=4$.
In this case, $N_{C_4}(G)=1$. 
If  $T_{n+4-s}$ is isomorphic to one of  $D_{n+2-n}$, $E_6$, $E_7$ and $E_8$, it yields that $N_{P_3}(T_{n+4-s})=n+3-s$  from Figure \ref{x1}.
In this case, we get $2N_{C_4}(C_s)=1$.
It contradicts the fact that $2N_{C_4}(C_s)$ is even.
So we have $C_s=C_4$ and $T_{n+4-s}=P_n$.

We derive that $G=C_4+P_{n}$ from the above discussion, i.e., $G$ is a non-isomorphic cospectral graphs of $\widetilde{D}_n$ if and only if $G= C_4+P_{n}$.
Similarly, we obtain that $G$ is a non-isomorphic cospectral graphs of $\widetilde{E}_6$ if and only if $G= C_6+K_1$.
\end{proof}

Next, we prove that there is an eigenvalue $\lambda$ of the power hypergraph $\widetilde{D}^{(k)}_n$ (resp. $\widetilde{E}^{(k)}_6$) but  $\lambda$ is not an eigenvalue of $(C_4+P_{n})^{(k)}$ (resp. $(C_6+K_1)^{(k)}$).
Therefore, we obtain the main result in this section.

\begin{thm}\label{pbjgjp}
All Smith's graphs are determined by the high-ordered spectrum.
\end{thm}
\begin{proof}
We will prove that $\widetilde{D}_n$ and $\widetilde{E}_6$ are DHS.
By Lemma \ref{zjyl}, we know that graph $G$ is a non-isomorphic cospectral graphs of $\widetilde{D}_n$ (resp. $\widetilde{E}_6$) if and only if $G= C_4+P_{n}$ (resp. $C_6+K_1$).
Then we only need to prove that $\widetilde{D}_n$ (resp. $\widetilde{E}_6$ ) is not high-ordered cospectral with $(C_4+P_{n})$ (resp. $C_6+K_1$ ).

Next, we prove that there is an eigenvalue $\lambda$ of the hypergraph $\widetilde{D}^{(k)}_n$ such that $\lambda$ is not an eigenvalue  of the hypergraph $(C_4+P_{n})^{(k)}$.
Since $D_n$ is an induced subgraph of $\widetilde{D}_n$ and $2\cos {\frac{1}{2n+2}}\pi$ is an eigenvalue of $D_n$ (Page 77 in \cite{cvetkovic1980spectra}), we know that $(2\cos {\frac{1}{2n+2}}\pi)^{\frac{2}{k}}$ is an eigenvalue of $\widetilde{D}^{(k)}_n$ by Lemma \ref{yinli1}.
However, we know that $(2\cos {\frac{1}{2n+2}}\pi)^{\frac{2}{k}}$ is not an eigenvalue of $(C_4+P_{n-3})^{(k)}$ by Lemma \ref{dingliben}.
Then  $\widetilde{D}_n$ is DHS.

Since $D_3$ is an induced subgraph of $\widetilde{E}_6$ and  $2\cos{\frac{\pi }{8}}$ is an eigenvalue of $D_3$, it yields that $(2\cos{\frac{\pi }{8}})^{\frac{2}{k}}$ is an eigenvalue of $\widetilde{E}^{(k)}_6$ by Lemma \ref{yinli1}.
However, we know that $(2\cos{\frac{\pi }{8}})^{\frac{2}{k}}$ is not an eigenvalue of $(C_6+K_1)^{(k)}$ by  Lemma \ref{dingliben}.
Then  $\widetilde{E}_6$ is DHS.
\end{proof}


\section{High-ordered cospectral invariants of trees}

Studying the high-ordered cospectral invariants of graph $G$ is helpful to study the high-ordered spectral characterizations of $G$.
In this section, we give some high-ordered cospectral invariants of trees by using the spectral moments of power hypertrees.
As their applications, we give infinitely many pairs of trees with same spectrum but different high-ordered spectrum.

Let $T$ and $T^*$ be two high-ordered cospectral trees.
Then $\mathrm{S}_d(T^{(k)})=\mathrm{S}_d({T^*}^{(k)})$ for all positive integer $k$ and $d$.
Let $\mathfrak{T}({m})$ denote the set of trees with $m$ edges.
From the spectral moment formula shown in Lemma \ref{yinlimcs}, we obtain
\begin{align}\label{fangcheng}
\sum\limits_{m=1}^{\frac{d}{k}}\frac{1}{2}{{(k - 1)}^{\left( {|E(T)| - m} \right)(k - 1)}}{k^{m(k - 2) + 1}} {\sum\limits_{\widehat{T} \in {\mathfrak{{T}}({m})}}{c_{\frac{2d}{k}}(\widehat{T})} } \left(N_{T}(\widehat{T})-N_{T^*}(\widehat{T})\right)=0
\end{align}for all positive integer $k\geq2$.
From Equation (\ref{fangcheng}), we obtain some high-ordered cospectral invariants of trees.

\begin{thm}\label{bbldl1}
If a tree $T$ is high-ordered cospectral with tree $T^*$, then
$${\sum\limits_{\widehat{T} \in {\mathfrak{T}(m)}}{c_{d}(\widehat{T})} } N_{T}(\widehat{T})={\sum\limits_{\widehat{T} \in {\mathfrak{T}({m})}}{c_{d}(\widehat{T})} } N_{T^*}(\widehat{T})$$
for all positive even integer $d$ and all $m \in [\frac{d}{2}] $.
\end{thm}

\begin{proof}
Since $T$ and $T^*$ are high-ordered cospectral, we have $\mathrm{S}_{d}(T^{(k)})=\mathrm{S}_{d}({T^*}^{(k)})$ for all
positive integer $k \geq 2$.
Since $\mathrm{S}_{d}(T^{(k)})=0$ if $k \nmid  d$, we assume $k \mid  d$ in the following proof.
Let $d=kz$.
From Equation (\ref{fangcheng}), we have
\begin{align}\label{shizixin}
\sum\limits_{m=1}^{z}\frac{1}{2}{{(k - 1)}^{\left( {|E(T)| - m} \right)(k - 1)}}{k^{m(k - 2) + 1}} {\sum\limits_{\widehat{T} \in {\mathfrak{T}({m})}}{c_{2z}(\widehat{T})} } \left(N_{T}(\widehat{T})-N_{T^*}(\widehat{T})\right)=0
\end{align}for all positive integer $k$.
Let $f_m(k)=\frac{1}{2}{{(k - 1)}^{\left( {|E(T)| - m} \right)(k - 1)}}{k^{m(k - 2) + 1}}$.
Let $y_m={\sum\nolimits_{\widehat{T} \in {\mathfrak{T}({m})}}{c_{2z}(\widehat{T})} } \left(N_{T}(\widehat{T})-N_{T^*}(\widehat{T})\right)$ for all $m \in [z]$.
From Equation (\ref{shizixin}), we have $\sum\nolimits_{m=1}^{z}{f_m(k)y_m}=0$.
Since $\frac{{{f_i}(k)}}{{{f_{i - 1}}(k)}} = \frac{{{k^{k - 2}}}}{{{{(k - 1)}^{k - 1}}}}$ for $i =2,3,\cdots,z $,
we get ${{f_m}(k)}={{f_1}(k)}(\frac{{{k^{k - 2}}}}{{{{(k - 1)}^{k - 1}}}})^{m-1}$.
Then $\sum\nolimits_{m=1}^{z}{{{f_1}(k)}(\frac{{{k^{k - 2}}}}{{{{(k - 1)}^{k - 1}}}})^{m-1}y_m}=0$.
For $z \geq 1$, take $k$ as $k_1,k_2,\ldots,k_z$, which are $z$ different positive integers.
It is following that
\begin{align}\label{fangchengzu}
\left[ {\begin{array}{*{20}{c}}
   {{f_1}({k_1})} & {{f_1}(k_1)}\frac{{{k_1^{k_1 - 2}}}}{{{{(k_1 - 1)}^{k_1 - 1}}}} &  \cdots  & {{f_1}(k_1)}(\frac{{{k_1^{k_1 - 2}}}}{{{{(k_1 - 1)}^{k_1 - 1}}}})^{z-1}  \\
   {{f_1}({k_2})} & {{f_1}(k_2)}\frac{{{k_2^{k_2 - 2}}}}{{{{(k_2 - 1)}^{k_2 - 1}}}} &  \cdots  & {{f_1}(k_2)}(\frac{{{k_2^{k_2 - 2}}}}{{{{(k_2 - 1)}^{k_2 - 1}}}})^{z-1}   \\
    \vdots  &  \vdots  &  \vdots  &  \vdots   \\
   {{f_1}({k_z})}  & {{f_1}(k_z)}\frac{{{k_z^{k_z - 2}}}}{{{{(k_z - 1)}^{k_z - 1}}}} &  \cdots  & {{f_1}(k_z)}(\frac{{{k_z^{k_z - 2}}}}{{{{(k_z - 1)}^{k_z - 1}}}})^{z-1}   \\
\end{array}} \right]\left[ {\begin{array}{*{20}{c}}
   {{y_1}}  \\
   {{y_2}}  \\
    \vdots   \\
   {{y_z}}  \\
\end{array}} \right] = 0.
\end{align}
Then the coefficient matrix of the Equation (\ref{fangchengzu}) is a Vandermonde matrix.
And since $k_1,k_2,\ldots,k_z$ are distinct,  the determinant of the  coefficient matrix are not equal to zero.
Then $y_m=0$ for all $m \in [z]$.
\end{proof}

Let $m=|E(T)|=|E(T^*)|$ in Theorem \ref{bbldl1}, we directly get the following high-ordered cospectral invariants.
\begin{thm}\label{pjxsbbl}
Let $T$ and $T^*$ be two high-ordered cospectral trees. Then $c_d(T)=c_d(T^*)$ for any positive even number $d \geq {2|E(T)|}$.
\end{thm}

Let the set of trees with $m$ edges ${{\mathfrak{T}}({m})}=\{{\widehat{T}}_1,{\widehat{T}}_2,\ldots,{\widehat{T}}_{|{{\mathfrak{T}}({m})}|}\}$. 
From Theorem \ref{bbldl1}, we get the following equation after taking $d$ as $d_1,d_2,\ldots,d_{|{\mathfrak{T}(m)}|}$.
\begin{align}\label{fangchengzu2}
\left[ {\begin{array}{*{20}{c}}
   {{c_{{d_1}}}({\widehat{T}_1})} & {{c_{{d_1}}}({\widehat{T}_2})} &  \cdots  & {{c_{{d_1}}}({\widehat{T}_{|{\mathfrak{T}(m)}|}})}  \\
   {{c_{{d_2}}}({\widehat{T}_1})} & {{c_{{d_2}}}({\widehat{T}_2})} &  \cdots  & {{c_{{d_2}}}({\widehat{T}_{|{\mathfrak{T}(m)}|}})}  \\
    \vdots  &  \vdots  &  \vdots  &  \vdots   \\
   {{c_{{d_{|{\mathfrak{T}(m)}|}}}}({\widehat{T}_1})} & {{c_{{d_{|{\mathfrak{T}(m)}|}}}}({\widehat{T}_2})} &  \cdots  & {{c_{{d_{|{\mathfrak{T}(m)}|}}}}({\widehat{T}_{|{\mathfrak{T}(m)}|}})}  \\
\end{array}} \right]\left[ {\begin{array}{*{20}{c}}
   {{h_1}}  \\
   {{h_2}}  \\
    \vdots   \\
   {{h_{|{\mathfrak{T}(m)}|}}}  \\
\end{array}} \right] = 0,
\end{align}
where $h_i=N_{T}(\widehat{T}_i)-N_{T^*}(\widehat{T}_i)$ for all $i \in [|{\mathfrak{T}(m)}|]$.
If there exist $d_1,d_2,\ldots,d_{|{\mathfrak{T}(m)}|}$ such that the coefficient matrix of  Equation (\ref{fangchengzu2}) is nonsingular,
we get $N_{T}(\widehat{T})=N_{T^*}(\widehat{T})$ for all $\widehat{T} \in \mathfrak{T}(m)$.
By the formula for the spectral moment coefficients for trees, i.e. Equation (\ref{pjxsgs}),
we can calculate the spectral moment coefficients $c_{d}({\widehat{T}})$.
We calculate the $d$-th order spectral moment coefficients of trees with 3 edges for $d=6,8$, the $t$-th order spectral moment coefficients of trees with 4 edges for $t=8,10,12$ and the $l$-th order spectral moment coefficients of trees with 5 edges for $l=10,12,14,16,18,20$ (see Table \ref{4xishu} and Table \ref{5xishu}).

\begin{figure}[H]
  \centering
  \subfloat[$P_2$]{
\begin{minipage}{0.10\textwidth}
\centering
\includegraphics[scale=0.2]{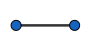}
\end{minipage}
}
 \subfloat[$P_3$]{
\begin{minipage}[t]{0.10\textwidth}
\centering
\includegraphics[scale=0.2]{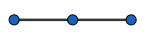}
\end{minipage}
}
 \subfloat[$P_4$]{
\begin{minipage}[t]{0.10\textwidth}
\centering
\includegraphics[scale=0.2]{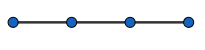}
\end{minipage}
}
  \subfloat[$S_4$]{
\begin{minipage}[t]{0.1\textwidth}
\centering
\includegraphics[scale=0.2]{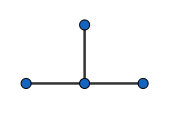}
\end{minipage}
}
   \subfloat[$P_5$]{
\begin{minipage}[t]{0.11\textwidth}
\centering
\includegraphics[scale=0.2]{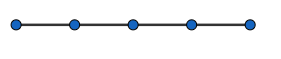}
\end{minipage}
}
   \subfloat[$Q_5$]{
\begin{minipage}[t]{0.10\textwidth}
\centering
\includegraphics[scale=0.2]{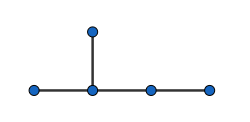}
\end{minipage}
}
   \subfloat[$S_5$]{
\begin{minipage}[t]{0.10\textwidth}
\centering
\includegraphics[scale=0.2]{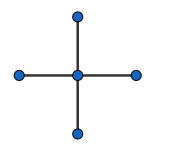}
\end{minipage}
} \\
   \subfloat[$P_6$]{
\begin{minipage}[t]{0.14\textwidth}
\centering
\includegraphics[scale=0.2]{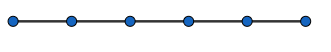}
\end{minipage}
}
   \subfloat[$Q_6$]{
\begin{minipage}[t]{0.12\textwidth}
\centering
\includegraphics[scale=0.2]{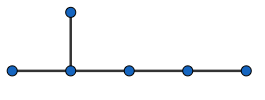}
\end{minipage}
}
   \subfloat[$R_6$]{
\begin{minipage}[t]{0.15\textwidth}
\centering
\includegraphics[scale=0.2]{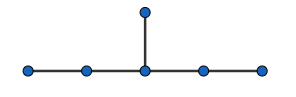}
\end{minipage}
}
   \subfloat[$H_6$]{
\begin{minipage}[t]{0.10\textwidth}
\centering
\includegraphics[scale=0.2]{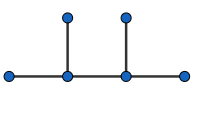}
\end{minipage}
}
   \subfloat[$J_6$]{
\begin{minipage}[t]{0.10\textwidth}
\centering
\includegraphics[scale=0.2]{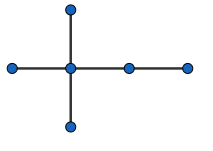}
\end{minipage}
}
   \subfloat[$S_6$]{
\begin{minipage}[t]{0.10\textwidth}
\centering
\includegraphics[scale=0.2]{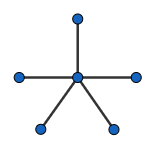}
\end{minipage}
} \\
 \caption{Trees with edges at most $5$}
 \end{figure}

\begin{table}[H]
  \centering
   \subfloat[The spectral moment coefficients of trees with 3 edges]{
  \begin{tabular}{|c|c|c|} \hline
 & $c_d(P_4)$&$c_d(S_4)$\\ \hline
$d=6$&6&12 \\ \hline
$d=8$&32&72 \\ \hline
\end{tabular} }~~~
\subfloat[The spectral moment coefficients of trees with 4 edges]{
\begin{tabular}{|c|c|c|c|} \hline
 & $c_d(P_5)$& $c_d(Q_5)$&$c_d(S_5)$\\ \hline
$d=8$&8&16&48 \\ \hline
$d=10$&60&140&480 \\ \hline
$d=12$&300&804&3120 \\ \hline
\end{tabular}}
  \caption{The spectral moment coefficients of trees with 3 or 4 edges}\label{4xishu}
\end{table}

\begin{table}[H]
  \centering
 \begin{tabular}{c|c|c|c|c|c|c} \hline
 & $c_d(P_6)$& $c_d(Q_6)$&$c_d(R_6)$&$c_d(H_6)$&$c_d(J_6)$&$c_d(S_6)$\\ \hline
$d=10$& 10 & 20 & 20 & 40 & 60 & 240\\ \hline
$d=12$& 96 & 216 & 228 & 504 & 792 & 3600\\ \hline
$d=14$& 588 & 1484 & 1652 & 3976 & 6552 & 33600\\ \hline
$d=16$& 2944 & 8304 & 9728 & 25216 & 43680 & 252000\\ \hline
$d=18$& 13158 & 41328 & 50832 & 140832 & 257184 & 1668240\\ \hline
$d=20$& 54730 & 190800 & 245880 & 724320 & 1398600 & 10206000\\ \hline
\end{tabular}
  \caption{The spectral moment coefficients of trees with 5 edges}\label{5xishu}
\end{table}

We obtain the following higher-ordered cospectral invariants about the number of subtrees.

\begin{thm}\label{zhuyao}
If a tree $T$ is high-ordered cospectral with a tree $T^*$,  then $N_{T}(\widehat{T}) =N_{T^*}(\widehat{T})$ for any tree $\widehat{T}$ within 5 edges.
\end{thm}
\begin{proof}
Since $T$ and $T^*$ are high-ordered cospectral, we know that $T$ and $T^*$ are cospectral.
Then  $N_{T}(P_2) =N_{T^*}(P_2)$ and $N_{T}(P_3) =N_{T^*}(P_3)$.
From Table \ref{4xishu} and Table \ref{5xishu}, we get the coefficient matrix of Equation (\ref{fangchengzu2}) when $m=3,4,5$.
It is easy to check that these coefficient matrices are nonsingular.
Then $N_{T}(\widehat{T})=N_{T^*}(\widehat{T})$ for all $\widehat{T} \in \mathfrak{T}(m)$, $m=1,2,3,4,5$.
\end{proof}

By the above high-ordered cospectral invariants, we get infinitely many pairs of cospectral trees with different high-ordered spectrum.
In \cite{schwenk1973almost}, Schwenk gave a useful method to construct cospectral trees and proved his famous conclusion:``Almost all trees are not DS".
Let $T_u$ and $T_v$ be the tree $T$ rooted at vertices $u$ and $v$, respectively.
For any rooted tree $F$,  the coalescences $F\cdot T_u$ and $F\cdot T_v$, as shown as in Figure \ref{tps}, are cospectral trees \cite{van2003graphs}.
We have $N_{F\cdot T_v}(R_6)-N_{F\cdot T_u}(R_6)=d$, where $d$ is the degree of the root of $F$.
By Theorem \ref{zhuyao}, we know that $F\cdot T_u$ is not high-ordered cospectral with $F\cdot T_v$.

\begin{pro}
Cospectral trees $F\cdot T_u$ and $F\cdot T_v$ have different high-ordered spectrum.
\end{pro}

\begin{figure}[H]
\centering
\subfloat[$F\cdot T_u$]{
\begin{minipage}[t]{0.45\textwidth}
\centering
\includegraphics[scale=0.55]{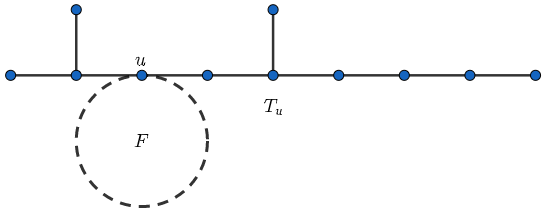}
\end{minipage}
}
\subfloat[$F\cdot T_v$]{
\begin{minipage}[t]{0.45\textwidth}
\centering
\includegraphics[scale=0.55]{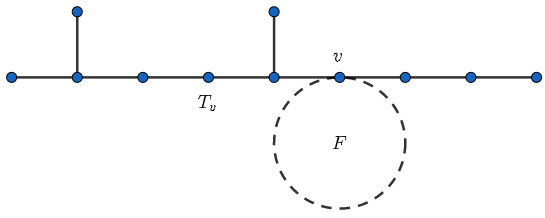}
\end{minipage}
}

\caption{The coalescences $F\cdot T_u$ and $F\cdot T_v$}\label{tps}
\end{figure}

\section{Conclusions and future work}

In this paper, we show that the high order spectrum of graphs can determine more graphs than the spectrum of graphs.
Up to now, we have not found any non-isomorphic high order cospectral graphs, that is to say, we have not found any graphs which are not DHS.
We are very interested in the existence and construction of non-isomorphic high order cospectral graphs.


The Reconstruction Conjecture of graphs is a famous conjecture in graph theory \cite{kelly1942isometric,ulam1960collection,harary1964reconstruction}.
In \cite{kelly1957congruence}, Kelly proved that ``The Reconstruction Conjecture" is true for trees.
This result is often called ``The Reconstruction Theorem for Trees".
By the Theorem \ref{zhuyao} in this paper, we know that the number of subtrees are high order cospectral invariants of trees.
From Kelly's Lemma \cite{kelly1957congruence}, we know that the number of subgraphs are reconstructible parameter of graphs.
We believe that the high order spectral characterizations of trees are closely related to the ``Reconstruction Theorem for Trees".
Based on the above introduction, we have the following conjecture.
\begin{con}
Two trees are isomorphic if and only if they have same high-ordered spectrum.
\end{con}

\section*{References}
\bibliographystyle{plain}
\bibliography{newHDSbib}
\end{spacing}
\end{document}